\documentclass[12pt, reqno]{amsart}

\title[Peripheral eigenvectors]{Peripherally automorphic unital completely positive maps}

\allowdisplaybreaks[1]

\usepackage{relsize}
\usepackage{amsmath,amsfonts,amssymb,amsthm,mathrsfs}
\usepackage{hyperref}
\usepackage{cleveref}
\usepackage[margin=2.5cm]{geometry}
\usepackage{setspace}
\usepackage{mathtools}

\usepackage{setspace}
\usepackage{xcolor}

\allowdisplaybreaks[1]

\numberwithin{equation}{section}
\newtheorem{theorem}{\bf Theorem}[section]

\newtheorem{cor}[theorem]{\bf Corollary}
\newtheorem{remark}[theorem]{\bf Remark}
\newtheorem{prop}[theorem]{\bf Proposition}
\newtheorem{defin}[theorem]{\bf Definition}
\newtheorem{example}[theorem]{\bf Example}

\newcommand{\seq}{\subseteq}

\newcommand{\N}{\mathbb{N}}
\newcommand{\M}{\mathbb{M}}

\newcommand{\T}{\mathbb{T}}
\newcommand{\Pt}{\mathcal{P}}

\newcommand{\ol}{\overline}
\newcommand{\mcal}{\mathcal}

\newcommand{\CS}{$C^*$-algebra}

\newcommand{\tr}{\text{tr}}

\makeatletter \@namedef{subjclassname@2020}{\textup{2020}
    Mathematics Subject Classification} \makeatother

\begin{document}

\author{B. V. Rajarama Bhat}
\address{Indian Statistical Institute, Stat Math Unit, R V College Post, Bengaluru, 560059, India.}
\email{bvrajaramabhat@gmail.com, bhat@isibang.ac.in}

\author{Samir Kar}
\address{Indian Statistical Institute, Stat Math Unit, R V College Post, Bengaluru, 560059, India.}
\email{msamirkar@gmail.com}

\author{Bharat Talwar}
\address{ Department of Mathematics, Nazarbayev University, Astana 010000, Kazakhstan.}
\email{btalwar.math@gmail.com, bharat.talwar@nu.edu.kz}



\maketitle

\begin{abstract}

We identify and characterize unital completely positive (UCP)  maps
on finite dimensional $C^*$-algebras for which the Choi-Effros
product extended to the space generated by peripheral eigenvectors
matches with the original product. We analyze a decomposition of
general UCP maps in finite dimensions into  persistent and transient
parts. It is shown that UCP maps on finite dimensional
$C^*$-algebras with spectrum contained in the unit circle are
$*$-automorphisms.
\end{abstract}

\section{Introduction}

It is well-known that the collection of fixed points of normal
unital completely positive (UCP) maps has a von Neumann algebra
structure on introducing a new product called the Choi-Effros
product. This von Neumann algebra is known as the Poisson boundary
of the UCP map. In \cite{BhatTalwarKar} it is noted that this
product can be extended to peripheral eigenvectors, that is to
eigenvectors with eigenvalue on the unit circle. As a consequence,
the norm closed linear span of peripheral eigenvectors becomes a
$C^*$-algebra. This was termed as peripheral Poisson boundary of the
UCP map. It is a natural question as to when does the new product of
the peripheral Poisson boundary agrees with the original product we
had in the beginning. In this article we answer this question for
UCP maps on matrix algebras.

There is extensive literature on the peripheral spectrum of positive
maps \cite{FallahRansford, GrohJoT, GrohLAA, GrohMZ, EvansKrohn}. A
lot of it is motivated by the classical Perron-Frobenius theory.
However we limit ourselves  to results directly relevant to the
problem posed above and  try to present these results in a unified
way. In the process, we also found some statements in the literature
which are inaccurate -see \Cref{error}.

We concentrate on UCP maps and their spectrum. Instead of
considering general finite dimensional $C^*$-algebera we are
focusing  on UCP maps on the full matrix algebra $\M_d(\mathbb{C}) :
= \M_d$. There is not much loss of generality in doing so, in view
of \Cref{ExtendToFinitedimensioanlC*-algebra}.

In rest of this section, we set up some basic notation and recall
some results from \cite{BhatTalwarKar}. Let $\M_d$ be the \CS \ of
all  $d \times d$ complex matrices and let $\tau : \M_d \to \M_d$ be
a UCP map. A non-zero vector $X$ in $\M _d$ satisfying $\tau
(X)=\lambda X$ for some $\lambda \in \mathbb{T}:=\{z\in
\mathbb{C}:|z|=1\}$ is called a peripheral eigenvector of $\tau .$
We have, $$E_{\lambda }(\tau ):=\{ X\in \M _d: \tau(X)=\lambda
X\}.$$ The peripheral space of $\tau $ is defined as
$$\Pt(\tau) := \overline{\text{span}}\{ X \in \M_d : \tau(X) = \lambda X \text{ for some } \lambda \text{ with }
|\lambda| = 1 \}.$$ Observe that $\Pt(\tau)$ is a vector subspace of
$\M _d$, containing the identity and it is closed under taking
adjoints. In other words, it is an operator system. In general
$\Pt(\tau)$ may not be closed under matrix multiplication. However
one may define a new product by setting (See \cite{BhatTalwarKar})
$$X\circ Y = s-\lim_{n\to \infty}(\lambda \mu )^{-n}\tau ^n(XY)$$ for peripheral
eigenvectors $X,Y$ with $\lambda$ and $\mu$ as their respective
eigenvalues. Without changing  the definitions of norm and adjoint,
$(\Pt(\tau), \circ)$ defines a \CS. Such a theory is possible for
normal UCP maps on general  von-Neumann algebras. Since our focus in
this paper is on the finite dimensional \CS \ $\M_d$, we may replace the norm closure `$\overline{\text{span}}$' by `$\text{span}$' and  the strong
operator limit above can be replaced with the norm limit. The
peripheral Poisson boundary of $\tau $ is defined as the
$C^*$-algebra $(\Pt(\tau) , \circ ).$ Given a UCP map one would like
to identify its peripheral Poisson boundary. The best case scenario
would be when the new product coincides with the original product of
$\Pt(\tau).$ Keeping this in mind we make the following definition.

\begin{defin}\label{PeripherallyAutomorphicDefinition}
Let $\tau :\M _d \to \M _d$ be a UCP map. Then $\tau $ is said to be
{\em peripherally automorphic\/} if $X\circ Y =XY$ for every $X, Y$ in $\Pt(\tau) .$
\end{defin}
If $X\circ Y=XY$ in $\Pt(\tau)$, we see that it is a $C^*$-algebra
under the original product and  as per  \cite[Theorem
2.12]{BhatTalwarKar}, restriction of $\tau$ to $\Pt(\tau)$ is a
$C\sp\ast$-automorphism. This justifies this definition. One of the
main goals of this article is to identify the collection of
peripherally automorphic UCP maps. A natural candidate is the
following well-studied class of maps.
\begin{defin}
Let $\tau : \M _d\to \M _d$ be a UCP map. Then $\tau $ is said to be
a {\em stationary} UCP map if it admits a faithful invariant state,
that is, there exists a faithful state $\phi $ on $\M _d$ such that
$$\phi (\tau (X))=\phi (X), ~~\forall X\in \M _d.$$
\end{defin}

We go through some of the existing literature on this class of UCP
maps.  All these maps are peripherally automorphic. Then we look at
the notion of multiplicative domains for UCP maps and finally
characterize the collection of peripherally automorphic UCP maps. In
the last section, we study the discrete dynamics of a UCP map on $\M
_d$ and its decomposition into persistent and transient parts. We
are able to demonstrate this decomposition for arbitrary UCP maps on
$\M _d.$ We use the familiar tool of Jordan decomposition. The main
disadvantage of considering the Jordan decomposition of UCP maps is
that it does not respect the crucial positivity/algebraic structures
of $\M _d.$ Here we are able to overcome this obstacle  to some
extent by appealing to results from the peripheral Poisson boundary
theory. For instance it shown that  UCP maps on finite dimensional
$C^*$-algebras with spectrum contained in the unit circle are
$*$-automorphisms. This partially settles a conjecture of F.
Fidaleo, F. Ottomano\ and\ S. Rossi \cite{Francesco2} in the
affirmative.

\section{Stationary maps and peripherally automorphic maps}

It is well-known that every completely positive map $\tau: \M_d \to
\M_d$ is of the form $\tau(X) = \sum_{i=1}^r L_i^\ast X L_i$ for
some  $d \times d$ matrices $L_i$ which are known as Kraus (or
Choi-Kraus) coefficients  of $\tau .$ As trace of a positive matrix
is zero if and only if the matrix is zero, for the normalized trace
$\tr$ on $\M_d$, the formula $\langle X,Y \rangle := \tr(X^\ast Y)$
defines an inner product for every $X,Y \in \M_d$ (Our inner
products are anti-linear in the first variable). Then the adjoint
$\tau^\ast$ of $\tau$ is the unique linear operator satisfying
$\langle \tau(X),Y \rangle = \langle X,\tau^\ast(Y) \rangle$ for
every $X,Y \in \M_d$.
For a completely positive map $\tau(X) = \sum_{i=1}^kL_i^\ast X L_i$
it can be  verified that $\tau^\ast(X) = \sum_{i=1}^r L_i X
L_i^\ast$.

\begin{remark}\label{adjointmap}
 A UCP map  $\tau :\M _d\to \M _d$ is stationary if and only if the adjoint $\tau^\ast$ has a strictly positive fixed point. A UCP
 map $\tau: \M _d\to \M _d$ is trace preserving if and only if $\tau^\ast$ is unital.
\end{remark}

The importance of mixed unitary channels is well understood.
They are all trace preserving. Going further there is the notion
called trace reducing maps in literature -see \cite{NAJ}. A linear
map $\tau: \M_d \to \M_d$ is called {\em trace reducing} if
$\tr(\tau(X)) \leq \tr(X)$ for every positive $X$ in $ \M_d$. In a
similar fashion one can call $\tau$ a {\em state reducing map}  if
there exists a state $\psi$ on $\M_d$ such that $\psi(\tau(X)) \leq
\psi(X)$ for every $0 \leq X \in \M_d$. The following observation
tells us that for unital linear maps these notions do not enlarge
the classes of maps under consideration. For this particular result
the state $\psi $ need not be faithful.

\begin{prop}\label{ReducingPreservingSame}
Let $\tau : \M_d \to \M_d$ be a unital linear map. Then $\tau$ is state reducing with respect to a  state $\psi $  if and only if $\tau$ is $\psi $ preserving.
\end{prop}
\begin{proof}

Suppose $\tau $ is state reducing for $\psi .$  Consider any $X$ in
$\M_d$ with $0\leq X\leq I$. Note that we also have $0\leq I-X$. So we
get $\psi (\tau (X))\leq \psi (X)$ and $\psi (\tau (I-X))\leq \psi
(I-X)$. Since $\tau , \psi $ are unital, the second inequality means $-\psi (\tau (X))\leq -\psi (X).$ Hence $\psi (\tau(X))= \psi
(X).$ The converse statement is a triviality.
\end{proof}

\begin{remark}\label{error}
If a UCP map preserves a faithful state then it has a range of
implications.
In \cite[Theorem 5]{Shemesh} several conditions were
claimed to be  equivalent to this property.  In our notation, this
theorem essentially reads as follows:  Let $\tau(X)=\sum_{i=1}^r
L_i^*XL_i$ be a UCP map on $\M_d$. Then the following are
equivalent.
\begin{enumerate}
\item $\tau$ is stationary.
\item The spaces $\Pt(\tau)$ and $E_1(\tau)$ corresponding to $\tau$ are $*$-algebras.
\item For $X\in \M_d$ and $\lambda \in \mathbb{T}$, $\tau (X)= \lambda X$ if
and only if $XL_i=\lambda L_iX$ for every $i.$
\item The algebra  generated by the Kraus coefficients $\{L_i: 1 \leq i \leq r\}$ is a $\ast$-algebra.
\end{enumerate}
Below  in \Cref{ShemeshWasWrong} we provide a UCP map $\tau$ such
that the spaces $\Pt(\tau)$ and $E_1(\tau)$ are $*$-algebras but
neither is $\tau$ a stationary map nor does the algebra generated by
Kraus coefficients form a $*$-algebra. In other words, (2) does not
imply (1) or (4).
\Cref{EquivalentConditionsForExistenceOfSpecialFaithfulState} shows
(1) is equivalent to (4) and with the help of  \Cref{Ourclass} and
\Cref{MainTheorem} we will see that (2) is equivalent to (3).
\end{remark}

\begin{example}\label{ShemeshWasWrong}
Consider the UCP map $\tau :\M _2\to \M _2$ defined by
    $$\tau \left(\begin{bmatrix}
        x_{11} & x_{12}\\
        x_{21} & x_{22}
    \end{bmatrix}\right)=\begin{bmatrix}
        \frac{x_{11}+ x_{22}}{2} &0\\
        0 &x_{22}
    \end{bmatrix},~~\forall \begin{bmatrix}
        x_{11} &x_{12}\\
        x_{21} &x_{22}
    \end{bmatrix} \in \M _2.$$  Here we may take the Kraus coefficients $$L_1:=\begin{bmatrix}
        \frac{1}{\sqrt 2} &0\\
        0 &0
    \end{bmatrix}, ~~L_2:=\begin{bmatrix}
        0 &0\\
        \frac{1}{\sqrt 2} &0
    \end{bmatrix}, ~~L_3:=\begin{bmatrix}
        0 &0\\
        0 &1
    \end{bmatrix}.$$
In this case $\Pt(\tau)$ and $E_1(\tau)$ are the $\ast$-algebra
$\{\lambda I:\lambda\in \mathbb{C}\}$. The adjoint map $\tau ^*$ is
given by
 $$\tau^*\left(\begin{bmatrix}
    x_{11} & x_{12}\\
    x_{21} & x_{22}
\end{bmatrix}\right)=\begin{bmatrix}
    \frac{x_{11}}{2} &0\\
    0 &x_{22}+\frac{x_{11}}{2}
\end{bmatrix}.$$
Clearly $\tau^*$ has no strictly positive fixed point. Hence from
\Cref{adjointmap}, $\tau $ is not stationary.
Also note that the algebra generated by the Kraus coefficients is the
algebra of all $2\times 2$  lower triangular matrices which is not a
$*$-algebra.
\end{example}

Now we have a theorem characterizing UCP maps preserving a faithful state, i.e, the stationary maps.
Some of these results are
well-known and are here just for completeness. Recall from
\cite{EvansKrohn} that a UCP map $\tau$ on a finite dimensional \CS
\ $A$ is said to be irreducible if there does not exist any
projection $P \in A$, different from $0,I$  such that $\tau (P)\leq P.$
Observe that (3) implies (1) of the following  theorem  shows that  every irreducible UCP map has a faithful invariant state.

\begin{theorem}\label{EquivalentConditionsForExistenceOfSpecialFaithfulState}
Let $\tau$ be a UCP map on $\M_d$ given by $\tau(X)=\sum_{i=1}^{r}
L_i^*XL_i$. Then the following are equivalent.
\begin{enumerate}
\item $\tau$ is stationary, that is, $\tau $ admits a faithful
invariant state.
\item  $\tau(P) \leq P$ for a projection $P$ implies that $\tau(P) =
P$.
\item There exists a set of mutually orthogonal projections  $\{ P_1, P_2, \ldots , P_l \}$ (for some $l$) in $\M_d$,
with $P_1+P_2+\cdots +P_l=I$,
 such that  $\tau_{j} := \tau_{|_{P_j \M_d P_j}}$ is an irreducible UCP map for every $1 \leq j \leq l$.
\item  The algebra  generated by Kraus coefficients $\{L_i\}_{1 \leq i \leq r}$ is a $\ast$-algebra.
 \end{enumerate}
\end{theorem}
\begin{proof}
(1) implies (2): For a stationary map $\tau$, we fix $\phi$ as a
faithful state on $\M_d$ satisfying $\phi(\tau(X)) = \phi(X)$ for
every $X \geq 0$.
Let $P$ be a projection such that $\tau (P)\leq
P$. On applying $\phi $, we get $\phi (P-\tau (P))=\phi (P)-\phi
(P)=0$. As $\phi $ is faithful, $P=\tau(P)$ follows.

(2) implies (3): By our hypothesis,
\begin{eqnarray*}
 \mcal{P} &:= &\{ P \in \M_d : P \text{ is a projection and } \tau(P)
= P\}\\& =&  \{ P \in \M_d : P \text{ is a projection and } \tau(P)
\leq P\}.\end{eqnarray*} Now if $\mcal{P}$ is trivial, $\tau $ is
irreducible and we are done. If not, using finite dimensionality of
$\M _d$, we can get a non-trivial minimal projection $P_1$ such that
$\tau (P_1)=P_1.$ Clearly, $\tau$  maps ${P_1 \M_d P_1}$ to ${P_1
\M_d P_1}$  and $\tau_{|_{P_1 \M_d P_1}}$ is an irreducible UCP map.
The unitality of $\tau $, yields $\tau (P_1^\perp) = P_1^\perp:= I -
P_1 $ and we can repeat the process on ${P_1^\perp  \M_d P_1^\perp
}$. Continuing this we can exhaust whole of $\M _d.$

(3) implies (4): Since $\tau_j$ is an irreducible UCP map on the \CS \
$P_j \M_d P_j$, there exists state $\phi_j$ on $P_j \M_d P_j$ such
that $\tau_j$ preserves $\phi_j$. Define a faithful linear
functional on $\M_d$ by $$\phi(X) :=  \sum_{j=1}^{l} \phi_j(P_j X
P_j).$$ Then $\phi$ is a state because $\phi(I) = \sum_{j=1}^{l}
\phi_j(P_j I P_j) = \sum_{j=1}^{l} \phi_j(P_j) = \sum_{j=1}^{l} P_j
= I$. It is well-known, and can be easily verified using the two definitions of multiplicative domain given later in this paper, that $\tau$ has block decomposition of the form
$\tau(P_j X P_j) =  \tau(P_j) \tau(X) \tau(P_j) = P_j \tau(X) P_j$ for every $1 \leq j \leq l$ and $X \in \M_d$. Hence \[ \phi(\tau(X)) =  \sum_{j=1}^{l} \phi_j(P_j \tau(X) P_j) =  \sum_{j=1}^{l} \phi_j(\tau(P_j X P_j)) = \sum_{j=1}^{l} \phi_j(P_j X P_j) = \phi(X).\]
Thus, $\tau $ is a stationary UCP map and hence we can
apply the trick used later in proving (4) implies (5) of \Cref{Ourclass}
along with the fact that $\tau(P_j) = P_j$ to obtain $P_jL_i = L_i
P_j$. If $\mathbb{A}$ denotes the algebra generated by $\{ L_i \}_{1
\leq i \leq r}$, then $\mathbb{A}P_j=P_j\mathbb{A}=P_j\mathbb{A}P_j$
and hence one may write $\mathbb{A} = \sum_{j=1}^l P_j \mathbb{A} P_j$. It now suffices to prove that $\mathbb{A}P_j$ is a $\ast$-algebra.

For every fixed $1 \leq j \leq l$, we claim that $\mathbb{A}P_j =
B(\text{range}(P_j))$, the set of all  linear operators on
$\text{range}(P_j)$. To prove the claim we use the Burnside's
classical theorem \cite{Burnside}. In order to apply the theorem, we
need to show that the subalgebra $\mathbb{A}P_j$ is irreducible i.e.
there is no proper non-trivial subspace of $\text{range}(P_j)$ which
is left invariant by every element of $\mathbb{A}P_j$. Let $S
\seq \text{range}(P_j)$ be a subspace which is invariant under
$L_iP_j$ for every $1 \leq i \leq r$. Then $L_i S  = L_i P_j S \seq
S$ and hence $S$ is invariant under each $L_i$. If we prove that
$S$ is reducing for each $L_i$, then $P_S L_i = L_i P_S$ for every
$i$, $P_S$ being the projection onto $S$. This will imply that
$$\tau(P_S) = \sum_{i=1}^{r} L_i^*P_SL_i = \sum_{i=1}^{r}
L_i^*L_iP_S = P_S.$$ Since each $\tau_j$ is irreducible, we will
obtain that  either $S=\{0\}$ or $S= \text{range}(P_j)$ proving the
claim.

Finally, we show that $S$ is reducing for each $L_i$.
It is known from \cite[Theorem 1.1]{Kribs} that there exists a projecton $P$ such that $P\leq \tau(P)$ (respectively, $P\geq \tau(P)$) if and only if range$(P)$ is invariant under $L_i$ (respectively, $L_i^\ast$) for every $1 \leq i \leq r$.
As $S$ is invariant under each $L_i$, we have $P_S \leq \tau(P_S)$.
Since $\tau$ is stationary one may use $(2)$ to obtain that $P_S = \tau(P_S)$.
This implies that $S$ is invariant under $L_i^\ast$ for every $i$, proving that $S$ is reducing for each $L_i$.

(4) implies (1): See $(4) \implies (1)$ of \cite[Theorem 5]{Shemesh}.
\end{proof}
We will study peripheral eigenvectors for UCP maps with faithful
invariant states. It is convenient to consider a larger class described
below.
Recall from (\cite{Choi}, \cite[Remark 2.7]{Rahaman}) that the
multiplicative domain $\mathcal{M}_\tau$ of a UCP map $\tau$ is the $C\sp \ast$-subalgebra \[\{X\in \mathbb{M}_d:\tau(XY)=\tau(X)\tau(Y) \text{ and
}\tau(YX)=\tau(Y)\tau(X) \text{ for every } Y \in \M_d\}\] which is
same as \[\{X\in \mathbb{M}_d:\tau(XX^*)=\tau(X)\tau(X^*) \text{ and
} \tau(X^*X)=\tau(X^*)\tau(X)\}.\]
Some special cases of the following result are well-known (See \cite[Theorem 2.1]{Kribs} and \cite[Theorem
5]{Burgarth}) and the proof of (4) implies (5) uses a standard trick (See for instance \cite[Theorem 4.25]{Watrous}).

\begin{theorem}\label{Ourclass}
Let $\tau :\M _d\to \M _d$ be a UCP map with a Choi-Kraus
decomposition $\tau (X)= \sum _{i=1}^rL_i^*XL_i, ~~\forall X\in \M
_d.$
Then the following are equivalent.
\begin{enumerate}
\item $\Pt(\tau)\subseteq \mathcal{M}_\tau.$
\item  $\Pt(\tau) \subseteq \mathcal{M}_{\tau ^\infty}:=\bigcap
_{k=1}^{\infty} \mathcal{M}_{\tau ^k}$.
\item $\tau$ is peripherally automorphic.
\item For $\lambda \in \mathbb{T}$, $X \in E_\lambda(\tau)  ~~\mbox{implies}~~ \tau(X^\ast X) = X^\ast X.$
\item   For $\lambda \in \mathbb{T}$,  $Y\in E_\lambda(\tau)$ if and only if $YL_i=\lambda L_iY$ for every $1 \leq i \leq r$.
\end{enumerate}
\end{theorem}
\begin{proof}
(1) implies (2): Consider $X\in \Pt(\tau)$ such that $\tau(X)=\lambda X$ with $\lambda\in \mathbb{T}$.
As $\Pt(\tau)\subseteq \mathcal{M}_\tau$, for every $k\in \N$ we have $\tau^k(XX^*)=\lambda^kX\bar{\lambda}^kX^*=\tau^k(X)\tau^k(X^*)$ and $\tau^k(X^*X)=\bar{\lambda}^kX^*\lambda^kX=\tau^k(X^*)\tau^k(X)$.
This shows that $\Pt(\tau) \seq \mathcal{M}_{\tau^k}$ for every $k \in \N$ and hence $\Pt(\tau) \subseteq \mathcal{M}_{\tau^\infty}$.

(2) implies (3) :
As $\Pt(\tau)$ is spanned by peripheral eigenvectors of $\tau$, it suffices to
prove that for $\tau(X) = \lambda X$ and $\tau(Y) = \mu Y$ with
$|\lambda| = |\mu|= 1$ we have $X \circ Y = XY$.
Using the hypothesis that $\Pt(\tau) \seq \mcal{M}_{\tau^\infty}$, we obtain $$X \circ Y = s-\lim (\lambda \mu)^{-n} \tau^n(X Y) = s-\lim (\lambda \mu)^{-n} \tau^n(X) \tau^n(Y) = XY.$$

(3) implies (4) :  Let $X \in E_\lambda(\tau)$ and $\lambda \in \T$.
From \cite[Theorem 2.12]{BhatTalwarKar} we have $\tau(X^\ast \circ X) = \tau(X^\ast) \circ \tau(X) = X^\ast \circ X$.
Using the hypothesis that $\tau$ is peripherally automorphic, one obtains $\tau(X^\ast X)  = X^\ast  X$.

(4) implies (5) : Suppose $\lambda \in \mathbb{T}$, $Y\in \M _d$ and
$\tau (Y)=\lambda Y$.
Then
\begin{eqnarray*}
    \sum_{i=1}^r \left( Y L_i - \lambda L_i Y\right)^\ast \left( Y L_i - \lambda L_i Y \right)
    &=& \sum_{i=1}^r \left( L_i^\ast Y^\ast Y L_i - \lambda L_i^\ast Y^\ast  L_i Y - \ol{\lambda} Y^\ast L_i^\ast Y L_i +  Y^\ast L_i^\ast  L_i Y \right)\\
    &=&  \tau (Y^\ast Y) - \lambda \tau(Y^\ast)  Y - \ol{\lambda} Y^\ast \tau(Y) + Y^\ast  Y \\
    &=& \tau(Y^\ast Y) - Y^\ast  Y -  Y^\ast Y +  Y^\ast  Y \\
    &=& \tau(Y^\ast Y) -  Y^\ast  Y\\
    &=& 0.
\end{eqnarray*}
Consequently, $YL_i=\lambda L_iY$ for all $i$.
Conversely, let $|\lambda| = 1$ and $YL_i=\lambda L_iY$ for every $1 \leq i \leq r$.
It follows from the Kraus decomposition of $\tau$ that $Y\in E_\lambda(\tau)$.

(5) implies (1) : To prove that $\Pt(\tau) \seq M_\tau$, it suffices to prove that for $X \in E_\lambda(\tau)$ with $|\lambda| = 1$ we have $\tau(X^\ast X) = \tau(X^\ast) \tau(X)$ and $\tau(X X^\ast) =  \tau(X) \tau(X^\ast)$.
As $X \in E_\lambda(\tau)$, this is same as proving that $\tau(X^\ast X) = X^\ast X$ and $\tau(X X^\ast) =  X X^\ast$.
We know from the hypothesis that $X L_i =  \lambda L_iX$ and $X^\ast L_i = \ol{\lambda} L_i X^\ast$ for every $1 \leq i \leq r$.
Hence, $$X^\ast X L_i = X^\ast \lambda L_i X = \lambda \ol{\lambda} L_i X^\ast X = L_i X^\ast X$$ and $$ X  X^\ast L_i = X \ol{\lambda} L_i X^\ast = \lambda \ol{\lambda} L_i  X X^\ast = L_i XX^\ast $$  establishing that $\tau(X^\ast X) = X^\ast X$  and $\tau(X X^\ast) =  X X^\ast$.
\end{proof}

\begin{cor}
Let $\tau$ be  peripherally automorphic, $X \in E_\lambda(\tau)$
with $\lambda \in \T$ and let $Y \in \M_d$ be such that $\tau(Y)=\mu
Y$. Then $XY\in E_{\lambda \mu}(\tau)$.
\end{cor}
\begin{proof}
Since $X \in \Pt(\tau) \seq \mathcal{M}_\tau$, we have $\tau(XY)=\tau(X)\tau(Y)=\lambda \mu XY$.
\end{proof}

\begin{cor}\label{SIsInR}
Suppose  $\tau :\M _d \to \M _d$ is a stationary UCP map with faithful invariant state $\phi$.
Then $\tau$ is peripherally automorphic.
\end{cor}
\begin{proof}
Let $\tau$ be stationary and $X \in E_\lambda(\tau) \text{ with } \lambda \in \mathbb{T}$.
Then $\tau(X) = \lambda X$, $\tau(X^\ast) = \ol{\lambda} X$ and using  Kadison-Schwarz inequality we obtain $\tau(X^\ast X) \geq \tau(X^\ast) \tau(X) = X^\ast X$.
An application of the faithful state $\phi$ now shows that $\tau(X^\ast X)  = X^\ast X$.
\end{proof}

The following result is particularly useful when we consider convex
sets of peripherally automorphic maps, such as  the class of trace
preserving maps (which includes mixed unitary channels) or more
generally stationary UCP maps with respect to a fixed faithful
state.
\begin{cor}\label{convexity}
Suppose $\tau = \sum _{j=1}^kp_j\tau _j$ for  some UCP maps $\tau_j$,  $p_j>0$ for $1 \leq j \leq k$ and $\sum_{j=1}^k p_j = 1$.
If $\tau $ is peripherally automorphic, then the peripheral spectrum of $\tau $ is contained in the intersection of the peripheral spectrums of $\tau _j$'s and $\Pt(\tau) \seq \Pt(\tau_j)$ for every $1 \leq j \leq k$.
\end{cor}
\begin{proof}
We make use of part (5) of Theorem \ref{Ourclass}.
If $\lambda $ is in the peripheral spectrum of $\tau $, with eigenvector $X$, then $XL=\lambda LX$ for every Kraus coefficient $L$ of $\tau$.
The Kraus coefficients of $\tau _j$'s can also be considered as scalar multiples of Kraus coefficients of $\tau$.
Thus, $XL=\lambda LX$ for every Kraus coefficient $L$ of $\tau_j$.
Using this and the expression for $\tau_j$'s in terms of their Kraus coefficients, one obtains that  $\lambda $ is in the peripheral spectrum of every $\tau_j$ with $X$ as corresponding eigenvector.
\end{proof}
The preceding corollary and its proof tells us that $\tau$ exhibits an extreme behavior on $\Pt(\tau)$ in the sense that for such $\tau$ and $\tau_j$'s we have $\tau(X) = \tau_j(X)$ for every $X \in \Pt(\tau)$ and $1 \leq j \leq k$.
We point it out here that in general there exists stationary maps which may be written as non-trivial convex combination of other stationary maps.

Now we have the main theorem of this Section.
\begin{theorem}\label{MainTheorem}
Let $\tau : \M _d\to \M _d$ be a UCP map.
Then $\tau$ is peripherally automorphic if and only if $\Pt(\tau)$ is closed under matrix multiplication. In this case $X\mapsto \tau (X)$  is a $C^*$-algebra automorphism on $\Pt(\tau)$.
\end{theorem}
\begin{proof}
If $\tau$ is peripherally automorphic, then \cite[Theorem 2.3]{BhatTalwarKar} implies that $\Pt(\tau)$ is closed under multiplication.
Conversely, let $\Pt(\tau)$ be closed under matrix multiplication.
Using \cite[Theorem 2.12]{BhatTalwarKar} we obtain that
$\tau(X^\ast \circ X) = \tau(X^\ast) \circ \tau(X) = X^\ast \circ X$ for $\lambda \in \mathbb{T}$ and  $X \in E_\lambda(\tau)$.
In view of \Cref{Ourclass}, it suffices to prove that $X^\ast \circ X = X^\ast X$.
As $X^\ast X \in \Pt(\tau)$, one may write $X^\ast X = \sum X_i$ where $X_i \in E_{\lambda_i}(\tau)$ with each $|\lambda_i| = 1$. From \cite[Theorem 2.5 \& Lemma 2.2]{BhatTalwarKar} we obtain a sequence $k_n$ of natural numbers such that $$X^\ast \circ X = s-\lim \tau^{k_n}(X^\ast X) = s-\lim \sum \lambda_i^{k_n} X_i = \sum X_i = X^\ast X.$$
That $\tau_{|_{\Pt(\tau)}}$ is a $C^*$-automorphism is now a consequence of \cite[Theorem 2.12]{BhatTalwarKar}.
\end{proof}

It is not difficult to see that direct sum of two peripherally automorphic maps is peripherally automorphic.
Corollary \ref{convexity} suggests that perhaps one should study
convex combinations of peripherally automorphic maps. However, below
we show that convex combinations and compositions of two
peripherally automorphic maps may not be peripherally automorphic,
unlike stationary maps with respect to a fixed state.

\begin{example}\label{AverageIsNotPeripherallyAutomorphic}\label{CompositionofPAIsNotPA}
    \begin{enumerate}
        \item For two peripherally automorphic maps $\tau_1$ and $\tau_2$ on $\mathbb{M}_3$ given by $$\tau_1\left( \begin{bmatrix}
            x_{11} &x_{12} &x_{13}\\
            x_{21} &x_{22} &x_{23}\\
            x_{31} &x_{32} &x_{33}
        \end{bmatrix} \right) = \begin{bmatrix}
            x_{11} &0 &0\\
            0 &x_{22} &0\\
            0 &0 &x_{11}
        \end{bmatrix} \text{ and } \tau_2\left(\begin{bmatrix}
            x_{11} &x_{12} &x_{13}\\
            x_{21} &x_{22} &x_{23}\\
            x_{31} &x_{32} &x_{33}
        \end{bmatrix} \right)= \begin{bmatrix}
            x_{11} &0 &0\\
            0 &x_{22} &0\\
            0 &0 &x_{22}
        \end{bmatrix}.$$
Then for $\tau= \frac{1}{2}(\tau_1+\tau_2)$ we have
$$\tau\left(\begin{bmatrix}
    1 & 0 & 0\\
    0& 3 & 0\\
    0 & 0 & 2
\end{bmatrix}\right)= \begin{bmatrix}
    1 & 0 & 0\\
    0& 3 & 0\\
    0 & 0 & 2
\end{bmatrix}$$ but $$\tau\left(\begin{bmatrix}
    1 & 0 & 0\\
    0& 3 & 0\\
    0 & 0 & 2
\end{bmatrix}^2\right) = \tau\left(\begin{bmatrix}
    1 & 0 & 0\\
    0& 9 & 0\\
    0 & 0 & 4
\end{bmatrix}\right) = \begin{bmatrix}
    1 & 0 & 0\\
    0& 9 & 0\\
    0 & 0 & 5
\end{bmatrix} \neq \begin{bmatrix}
    0 & 0 & 0\\
    0& 9 & 0\\
    0 & 0 & 4
\end{bmatrix}.$$
Thus, $\tau$ is not peripherally automorphic.

\item Let $$\tau_1\left(\begin{bmatrix}
            x_{11} &x_{12} &x_{13}\\
            x_{21} &x_{22} &x_{23}\\
            x_{31} &x_{32} &x_{33}
        \end{bmatrix}\right) = \begin{bmatrix}
            x_{11} &0 &0\\
            0 &x_{22} &0\\
            0 &0 &\frac{x_{11}+x_{33}}{2}
        \end{bmatrix}$$  and  $$\tau_2\left(\begin{bmatrix}
            x_{11} &x_{12} &x_{13}\\
            x_{21} &x_{22} &x_{23}\\
            x_{31} &x_{32} &x_{33}
        \end{bmatrix}\right) =  \begin{bmatrix}
            x_{11} &0 &0\\
            0 &x_{22} &0\\
            0 &0 &\frac{x_{22}+x_{33}}{2}
        \end{bmatrix}.$$
Here, $$\Pt(\tau_1)=\left\{ \begin{bmatrix}
            a &0 &0\\
            0 &b &0\\
            0 &0 &a
        \end{bmatrix}:a,b\in \mathbb{C}\right\} \text{ and }
        \Pt(\tau_2)=\left\{ \begin{bmatrix}
            b &0 &0\\
            0 &a &0\\
            0 &0 &a
        \end{bmatrix}:a,b\in
        \mathbb{C}\right\}.$$
It follows from \Cref{MainTheorem} that $\tau_1$ and $\tau_2$ are peripherally automorphic maps.
It is verified that their composition $$\tau := (\tau_2\circ\tau_1) \left(\begin{bmatrix}
    x_{11} &x_{12} &x_{13}\\
    x_{21} &x_{22} &x_{23}\\
    x_{31} &x_{32} &x_{33}
\end{bmatrix}\right) = \begin{bmatrix}
    x_{11} &0 &0\\
    0 &x_{22} &0\\
    0 &0 &\frac{x_{11}+x_{33}+ 2 x_{22}}{4}
\end{bmatrix}$$
is not peripherally automorphic since
$$\tau\left(\begin{bmatrix}
    3 & 0 & 0\\
    0& 0 & 0\\
    0 & 0 & 1
\end{bmatrix}\right)= \begin{bmatrix}
    3 & 0 & 0\\
    0& 0 & 0\\
    0 & 0 & 1
\end{bmatrix}$$ but $$\tau\left(\begin{bmatrix}
    3 & 0 & 0\\
    0& 0 & 0\\
    0 & 0 & 1
\end{bmatrix}^2\right) = \tau\left(\begin{bmatrix}
    9 & 0 & 0\\
    0& 0 & 0\\
    0 & 0 & 1
\end{bmatrix}\right) = \begin{bmatrix}
    9 & 0 & 0\\
    0& 0 & 0\\
    0 & 0 & \frac{5}{2}
\end{bmatrix} \neq \begin{bmatrix}
    9 & 0 & 0\\
    0& 0 & 0\\
    0 & 0 & 1
\end{bmatrix}.$$
\end{enumerate}
\end{example}

\section{ A decomposition theorem for  UCP maps}

Consider the discrete dynamics $\{ \tau ^n: n\geq 0\}$ of a UCP map on $\M _d$. It is clear that peripheral eigenvectors have  periodic
or persistent behavior and other eigenvectors decay to zero under
this dynamics. So we may expect that $\M _d$ has some kind of direct
sum decomposition consisting of a persistent part and a transitive
part.
One may also expect that on the persistent part, $\tau$ acts as an automorphism.
However, it is to be kept in mind that in general UCP maps are
not diagonalizable and so eigenvectors may not describe the full
picture. There is extensive literature on this subject as this has a
lot of physical significance. Most papers obtain such a
decomposition under some assumption or the other on UCP map (See
\cite[Theorem 2.5]{Rahaman},\cite[Theorem 9]{CSU}, \cite{FV}) such as irreducibility, existence of faithful invariant
state etc. In this Section we have such a result in full generality.
It is applicable to all UCP maps.

 \begin{theorem}\label{orthogonal}
 Let $\tau$ be a UCP map on $\mathbb{M}_d$. Then the domain $\M _d$ has a
 unique vector space direct sum decomposition:
 $$\M _d=\Pt(\tau ) \oplus \mcal{N}(\tau ),$$
where $\Pt(\tau )$ is the peripheral space of $\tau $  and  $\mcal{N}(\tau)=\{X\in \M_d:\lim_{n\to \infty}\tau ^n(X)=0\}$.
Furthermore, $\Pt(\tau ^m)=\Pt(\tau )$ and $\mcal{N}(\tau ^m)= \mcal{N}(\tau )$ for
every $m \geq 1.$
 \end{theorem}

\begin{proof}

From the  Jordan decomposition of $\tau$ we have
$$\mathbb{M}_d=\ker(\tau-\lambda_1\mathcal{I})^{\nu(\lambda_1)}\oplus \ker(\tau-\lambda_2\mathcal{I})^{\nu(\lambda_2)}\oplus \cdots \oplus \ker(\tau-\lambda_k\mathcal{I})^{\nu(\lambda_k)}$$
where $\mathcal{I}$ is the identity map on $\mathbb{M}_d$,
$\{\lambda_1,\lambda_2,\ldots, \lambda_k\}$ is the set of
eigenvalues of $\tau$ and $\nu(\lambda_i)$ is the algebraic
multiplicity of the eigenvalue $\lambda_i$. It is known that all the
Jordan blocks for peripheral eigenvalues of a UCP map are one
dimensional (see \cite{MW}, \cite{NAJ}). Hence the algebraic
multiplicity $\nu(\lambda)$ of the eigenvalue $\lambda \in \T$ is
same as its geometric multiplicity. Therefore  the persistent part
$\oplus_{\lambda\in \sigma(\tau):|\lambda|=1}\ker(\tau-\lambda
\mathcal{I})^{\nu(\lambda)}$ is same as $\Pt(\tau)$.
 For $X\in
\ker(\tau-\lambda_i\mathcal{I})^{\nu(\lambda_i)} $ with
$|\lambda_i|<1$, we have $\tau^n(X)\to 0$. Therefore taking $\mcal{N}(\tau) :=
\oplus_{\lambda\in \sigma(\tau):|\lambda|<1}\ker(\tau-\lambda
\mathcal{I})^{\nu(\lambda)}$, we have the required decomposition.

Suppose $Z\in \M _d$, satisfies $\lim _{n\to \infty }\tau ^n(Z)=0$,
and $Z=X+Y$ is the decomposition of $Z$, with $X\in \Pt(\tau )$, $Y\in
\mcal{N}(\tau )$. As $\lim _{n\to \infty }\tau ^n(Y)=0$, we get $\lim
_{n\to \infty}\tau ^n(X)=0$. But this is not possible unless $X=0$,
as $\tau $ is isometric on $\Pt(\tau )$. This proves the uniqueness of
the decomposition. The second part is clear from the Jordan
decomposition of $\tau .$
\end{proof}

In this theorem, the persistent part $\Pt(\tau )$ is always present as
$\tau $ is unital. The transient part $\mcal{N}(\tau )$ is absent if and
only if $\tau $ is an automorphism. In other words if the spectrum
of $\tau $ is contained in the unit circle then $\tau $ is an
automorphism. This combined with Remark
\ref{ExtendToFinitedimensioanlC*-algebra} settles Conjecture 5.5 of
  \cite{Francesco2} for
finite dimensional $C^*$-algebras.

 As mentioned before, it follows
from \cite[Theorem 2.12]{BhatTalwarKar} that the persistent part
$\Pt(\tau)$  has a $C\sp \ast$-algebraic structure $(\Pt(\tau), \circ)$
and $\tau$ restricted to this subspace becomes an automorphism. In
addition, if the algebraic structure on $\Pt(\tau)$ is coming from
matrix multiplication, then from \Cref{MainTheorem}, $\tau$ becomes
peripherally automorphic. One might ask whether there exists a
unitary $U \in \M_d$ such that $\tau_{_{\Pt(\tau)}}(X) = U X U^\ast$
for every $X \in \Pt(\tau)$. This is not true in general as
established by the following example.
\begin{example}\label{RestrictionIsNotComingFromAUnitary}
    Let $\tau:\mathbb{M}_3\to \mathbb{M}_3$ be the UCP map given by $$\tau\left(\begin{bmatrix}
        x_{11} &x_{12} &x_{13}\\
        x_{21} &x_{22} &x_{23}\\
        x_{31} &x_{32} &x_{33}
    \end{bmatrix}\right) = \begin{bmatrix}
        x_{33} &0 &0\\
        0 &x_{33} &0\\
        0 &0 &\frac{x_{11}+x_{22}}{2}
    \end{bmatrix}.$$
    Consider a faithful state $\phi$ given by $\phi((x_{ij}))= \frac{x_{11}}{4}+\frac{x_{22}}{4}+\frac{x_{33}}{2}$ which satisfies $\phi (\tau (X))=\phi(X)$ for every $X \in \mathbb{M}_3$.
    This shows that $\tau$ is stationary.
    The set of peripheral eigenvalues of $\tau$ is  $\{ 1,-1\}$ and $$\Pt(\tau)=\left\{\begin{bmatrix}
        a &0 &0\\
        0 &a &0\\
        0 &0 &b
    \end{bmatrix}:a,b\in \mathbb{C} \right\}.$$
It is not difficult to see that $\tau$ does not preserve trace of elements of $\Pt(\tau)$.
So the restriction of $\tau$ to $\Pt(\tau)$ can not be same as conjugation by any unitary in $\M_d$.
\end{example}
\begin{remark}
In \cite{FV} authors deal with  decoherence problem using projection
maps. Let $$Q_\tau:=\frac{1}{2\pi i}\int_{\gamma}
(\mathcal{I}-\tau)^{-1}d\lambda,$$ where  $\gamma$ is any Jordan
curve (i.e. a smooth and closed curve on a simply connected domain)
in the open unit disk containing all the non-peripheral eigenvalues
of $\tau$ in its bounded component. Let
$P_\tau:=\mathcal{I}-Q_\tau$. Therefore
$\mathbb{M}_d=P\tau\mathbb{M}_d\oplus Q_\tau \mathbb{M}_d$. Here
$P_\tau$ is an idempotent UCP map on $\mathbb{M}_d$ and hence
$(P_\tau\mathbb{M}_d,\circ)$ has $C\sp \ast$-algebraic  structure
under Choi-Effros product $\circ$. Not surprisingly
$P_\tau\mathbb{M}_d$ is same as $\Pt(\tau).$
    \end{remark}

\begin{theorem}
Let  $\tau :\M _d\to \M _d$ be a stationary UCP map with faithful
invariant state $\phi .$ Then $\M _d$ is a Hilbert space with inner
product $\langle X, Y\rangle :=\phi (X^*Y), X,Y\in \M _d$ and with
respect to this inner product the decomposition in Theorem
\ref{orthogonal} is an orthogonal decomposition.
\end{theorem}
\begin{proof}
As $\tau $ is a stationary UCP maps, from Corollary 2.8 and part (1)
 of Theorem 2.6, $\Pt(\tau )$ is in the multiplicative domain $M_\tau
$ of $\tau $.  Therefore for any $X, Y \in \M_d$, with $X\in \Pt(\tau
)$, we have
$$\langle X, Y \rangle =  \phi(X^\ast Y) =  \phi(\tau(X^\ast Y)) = \phi(\tau(X)^\ast \tau(Y))  = \langle \tau(X), \tau(Y) \rangle
.$$

Let $Y \in \mcal{N}(\tau ).$ Then,   $\lim _{n\to \infty}\tau^n(Y) = 0$.
However, $\langle X, Y \rangle = \langle \tau(X), \tau(Y) \rangle =
\langle \tau^2(X), \tau^2(Y) \rangle = \langle \tau^n(X), \tau^n(Y)
\rangle$ for every $n \in \N$. Now the fact that $\tau$ is
contractive  yields that $\langle X, Y \rangle = 0.$
\end{proof}

\begin{remark}\label{ExtendToFinitedimensioanlC*-algebra}
Suppose $\mathcal{A}$ is a finite dimensional unital $C^*$-algebra
and let $\tau : \mathcal{A}\to \mathcal{A}$ be a UCP map. Without
loss of generality, we may take $\mathcal{A}= \oplus _{j=1}^k\M
_{d_j}$ for some $d_1, \ldots , d_k, k\in \mathbb{N}.$ Let
$\mathcal{H}= \oplus _{j=1}^k(\mathbb{C}^{d_j})$ and let $P_j$ be
the projection of $\mathcal{H}$ to $\mathbb{C}^{d_j}$. Let
$\mathcal{B}$ be the $C^*$-algebra $\mathscr{B}(\mathcal{H})= \M
_{d}$ where $d= \sum _{j=1}^kd_j.$ Define $\tilde{\tau}:\M _d\to \M
_d$ by $\tilde{\tau}(X) = \tau ( \sum _{j=1}^dP_jXP_j)$. Then the
peripheral space of $\tilde{\tau }$ is identical to that of $\tau $.
This way, all the analysis done here can be extended to UCP maps on
general finite dimensional $C^*$-algebras.
\end{remark}

A UCP map $\tau : \M _d\to \M _d$ is said to be faithful if $\tau(X^\ast X)=0$ implies $X=0$. The following example shows that a faithful
UCP map need not be peripherally automorphic.

\begin{example}
Let $$\tau\left(\begin{bmatrix}
    x_{11} &x_{12} &x_{13}\\
    x_{21} &x_{22} &x_{23}\\
    x_{31} &x_{32} &x_{33}
\end{bmatrix}\right) = \begin{bmatrix}
    x_{11} &0 &0\\
    0 &x_{22} &0\\
    0 &0 & \frac{x_{11} + x_{22}+ x_{33}}{3}
\end{bmatrix}.$$
Then $\tau$ is a faithful UCP map on $\M_3$ such that $$\tau\left(\begin{bmatrix}
    0 & 0 & 0\\
    0& 2 & 0\\
    0 & 0 & 1
\end{bmatrix}\right)= \begin{bmatrix}
0 & 0 & 0\\
0& 2 & 0\\
0 & 0 & 1
\end{bmatrix}$$ but $$\tau\left(\begin{bmatrix}
0 & 0 & 0\\
0& 2 & 0\\
0 & 0 & 1
\end{bmatrix}^2\right) = \tau\left(\begin{bmatrix}
0 & 0 & 0\\
0& 4 & 0\\
0 & 0 & 1
\end{bmatrix}\right) = \begin{bmatrix}
0 & 0 & 0\\
0& 4 & 0\\
0 & 0 &
\frac{5}{3}
\end{bmatrix} \neq \begin{bmatrix}
0 & 0 & 0\\
0& 4 & 0\\
0 & 0 & 1
\end{bmatrix}.$$
Thus, $\tau$ is not peripherally automorphic.
\end{example}

\section*{Acknowledgments}
Bhat gratefully acknowledges funding from  SERB(India) through JC
Bose Fellowship No. JBR/2021/000024. Kar is thankful to NBHM (India) for funding.
Talwar thanks ISI Bangalore for financial support through Research Associate scheme.
A part of this work was completed during his research assistantship at Nazarbayev University and he appreciates their support.

\end{document}